\documentclass[a4paper, 11pt]{article}
\usepackage[scale=0.8,centering]{geometry}
\usepackage{amsmath,amssymb,amsthm,graphicx,bbm,pdfpages}
\usepackage{color,CJK,mathrsfs}
\usepackage{hyperref}
\usepackage{extarrows}
\usepackage[font={footnotesize}]{caption}

\usepackage[T1]{fontenc}

\usepackage{enumitem}
\setitemize{itemsep=1pt,topsep=2pt,parsep=1pt,partopsep=1pt}

\numberwithin{equation}{section}
\newtheorem{theorem}{\textbf{Theorem}}[section]

\newtheorem{lemma}[theorem]{\textbf{Lemma}}

\newtheorem{remark}{\textbf{Remark}}[section]

\newcommand{\bbE}{{\ensuremath{\mathbbm E}} }

\newcommand{\bbN}{{\ensuremath{\mathbbm N}} }

\newcommand{\bbP}{{\ensuremath{\mathbbm P}} }

\newcommand{\bbR}{{\ensuremath{\mathbbm R}} }

\newcommand{\bbT}{{\ensuremath{\mathbbm T}} }

\newcommand{\bbV}{{\ensuremath{\mathbbm V}} }

\newcommand{\bbZ}{{\ensuremath{\mathbbm Z}} }
\newcommand{\var}{\bbV\!\mbox{ar}}

\newcommand{\cA}{{\ensuremath{\mathcal A}} }

\newcommand{\cC}{{\ensuremath{\mathcal C}} }

\newcommand{\cZ}{{\ensuremath{\mathcal Z}} }

\newcommand{\bE}{{\ensuremath{\mathbf E}} }

\newcommand{\bP}{{\ensuremath{\mathbf P}} }

\newcommand{\bT}{{\ensuremath{\mathbf T}} }

\newcommand{\bZ}{{\ensuremath{\mathbf Z}} }

\newcommand{\be}{\begin{equation}}
\newcommand{\ee}{\end{equation}}
\newcommand{\bee}{\begin{equation*}}
\newcommand{\eee}{\end{equation*}}
\newcommand{\bc}{\be\begin{array}{r@{\,}c@{\,}l}}
\newcommand{\ba}{\begin{array}}
\newcommand{\ea}{\end{array}}
\newcommand{\ec}{\end{array}\ee}
\newcommand{\dis}{\displaystyle}

\newcommand{\dd}{\mbox{d}}

\newcommand{\di}{{\rm d}}

\definecolor{rosso}{RGB}{206,43,55}
\definecolor{blu}{RGB}{140,204,171}
\definecolor{verde}{RGB}{0,146,70}
\definecolor{arancione}{RGB}{255,102,51}
\definecolor{viola}{RGB}{255,0,255}

\title{Weak coupling limits for directed polymers in tube environments}

\author{Ran Wei 
\footnote{Department of Mathematics, Nanjing University, Nanjing 210093, China. Email: weiran@nju.edu.cn}
\and Jinjiong Yu 
\footnote{KLATASDS-MOE, School of Statistics, East China Normal University, 3663 North Zhongshan Road, Shanghai 200062, China.
	Email: jjyu@sfs.ecnu.edu.cn}
\footnote{NYU-ECNU Institute of Mathematical Sciences at NYU Shanghai, Shanghai, 200062, China}
}

\begin{document}
\maketitle
\begin{abstract}
In this paper, we study a model of directed polymers in random environment, where the environment is restricted to a time-space tube whose spatial width grows polynomially with time.
It can be viewed as an interpolation between the disordered pinning model and the classic directed polymer model.
We prove weak coupling limits for the directed polymer partition functions in such tube environments in all dimensions, as the inverse temperature vanishes at a suitable rate.
As the tube width varies, transitions between regimes of disorder irrelevance, marginal relevance and disorder relevance are observed.
\end{abstract}

\vspace{.5cm}


\section{Introduction}
\subsection{The directed polymer model}
In this paper, we study a model of \textit{directed polymers in random environment}. This model investigates the behavior of \textit{polymer chains} when they are stretched in some solvent with charges or impurities (called \textit{random environment}). The configurations of polymer chains are influenced by the interaction between polymers and environment. Whether this interaction changes qualitatively the behavior of polymer chains is an important question in the study of the directed polymer model. When the behavior of the model is changed qualitatively as the strength of interaction with the random environment increases, the system undergoes a phase transition, which is an important and interesting phenomenon in statistical physics.

The directed polymer model was first introduced by Huse and Henley in \cite{HH85} to study the domain walls in Ising systems. The first mathematical study on this model was done by Imbrie and Spencer in \cite{IS88}. Thereafter, the model became increasingly popular in mathematical physics community and attracted both probabilists and physicists. During the last thirty years, a large number of studies have been carried out. Readers may refer to \cite{C17} and references therein for more information on the directed polymer model.

For the classic directed polymer model, the random environment consists of i.i.d.\ random variables at each time-space lattice point. There is a related model, called the \textit{disordered pinning model}, where the random environment consists of i.i.d.\ random variables on the time-space line $\bbN\times\{0\}$. We consider a variant of the directed polymer model, where the random environment is restricted to a time-space tube, which interpolates between the disordered pinning model and the classic directed polymer model. The motivation is to investigate how the properties of the model change during transitions from the pinning model to the classic directed polymer model (see Subsection \ref{S12}).

We now introduce our model. Let $S:=(S_n)_{n\geq0}$ be a simple symmetric random walk on $\bbZ^d$ with $S_0=0$, representing the polymer chain. The law and expectation of $S$ are denoted by $\bP$ and $\bE$ respectively. The random environment (also called \textit{disorder}) is encoded by i.i.d.\ random variables $\omega:=(\omega_z)_{z\in\bbN\times\bbZ^d}$ indexed by time-space lattice sites. The law and expectation of $\omega$ are denoted by $\bbP$ and $\bbE$ respectively. We assume that $\omega$ has some finite exponential moments and denote its logarithmic moment generating function by
\begin{equation}\label{lambda}
\lambda(\beta):=\log\bbE[\exp(\beta\omega_z)]<\infty,\quad\forall~|\beta|<\beta_0~\text{for some}~\beta_0>0.
\end{equation}
Without loss of generality and for computational simplicity, we further assume that
\begin{equation}\label{Evarw}
\bbE[\omega_z]=0,\quad\bbE[\omega^2_z]=1.
\end{equation}

Our directed polymer model, up to time $N$, is then defined via a Gibbs transform
\begin{equation}\label{dp}
\frac{\dd\bP_{N,\beta}^{\omega}}{\dd\bP}(S):=\frac{1}{Z_{N,\beta}^{\omega}}\exp\left(\sum\limits_{n=1}^{N}(\beta\omega_{n,S_n}-\lambda(\beta))\mathbbm{1}_{\{(n,S_n)\in\Omega_N\}}\right),
\end{equation}
where $\beta$ denotes the \textit{inverse temperature}, $\Omega_N$ is a subset of $\bbN\times\bbZ^d$ where the random environment lies, $\bP_{N,\beta}^{\omega}$ is called the \textit{polymer measure}, and
\begin{equation}\label{Z}
Z_{N,\beta}^{\omega}=\bE\left[\exp\left(\sum\limits_{n=1}^{N}(\beta\omega_{n,S_n}-\lambda(\beta))\mathbbm{1}_{\{(n,S_n)\in\Omega_N\}}\right)\right]
\end{equation}  
is called the \textit{partition function}, which makes $\bP_{N,\beta}^{\omega}$ a (random) probability measure. Although this model is described by the probability measure $\bP_{N,\beta}^{\omega}$, we usually first study the partition function $Z_{N,\beta}^{\omega}$, which carries rich enough physical information for the system and is an  important step in the study of the polymer measure.

\subsection{Motivation of the paper}\label{S12} We first review the classic directed polymer model and the disordered pinning model, which motivate our work:

\noindent\textbf{$\bullet$ The classic $(1+d)$-dimensional directed polymer model:} The most standard and well-known definition of the directed polymer model is \eqref{dp} with $\Omega_N\equiv\bbN\times\bbZ^d$. That is, there is disorder at each time-space lattice point and the indicator in \eqref{dp} can thus be omitted.

\noindent\textbf{$\bullet$ The disordered pinning model:}
$\Omega_N\equiv\bbN\times\{0\}$, namely the disorder only lies on a \textit{defect line}. Since the distribution of $S$ is perturbed only when it returns to the origin, an alternative (and also more general) way to define the pinning model is by:
\begin{equation}\label{pinning}
\frac{\dd\bP_{N,\beta}^{\omega,h}}{\dd\bP}(\tau):=\frac{1}{Z_{N,\beta}^{\omega,h}}\exp\left(\sum\limits_{n=1}^{N}(\beta\omega_n-h)\mathbbm{1}_{\{n\in\tau\}}\right),
\end{equation}
where $h$ is an \textit{external field}, and $\tau=\{0,\tau_1,\tau_2,\ldots\}$ is a renewal process with
\begin{equation}\label{renewal}
K(n):=\bP(\tau_1=n)=\frac{L(n)}{n^{1+\alpha}},\quad\alpha\geq0\quad\text{and}\quad\sum\limits_{n=1}^{\infty}K(n)\leq1.
\end{equation}
Here $L(\cdot)$ is a slowly varying function (see \cite{BGT89} for reference), \textit{i.e.}, $L(x)>0$ on $(0,\infty)$ and for any $b>0$, $\lim_{x\to\infty}L(bx)/L(x)=1$. 
In this paper, we set $h=\lambda(\beta)$.

\begin{remark}\label{RWpin}
Clearly, the return times to $0$ for a simple random walk on $\bbZ^d$ give rise to a renewal process and hence define an associated pinning model. Indeed, it is known that in \eqref{renewal}, $K(2n)\sim c/n^{3/2}$ for $d=1$ $($\textit{cf.} \cite{BL18}$)$, $K(2n)\sim c/(n\log^{2}n)$ for $d=2$ $($\textit{cf.} \cite{AZ10}$)$, and $K(2n)\sim c/n^{d/2}$ for $d\geq3$ $($\textit{cf.} \cite{G11}$)$. Note that $\sum_{n=1}^{\infty}K(n)=1$ for $d=1,2$, while $\sum_{n=1}^{\infty}K(n)<1$ for $d\geq3$, due to the recurrence/transience of the underling walk.
\end{remark}

We have mentioned in the previous subsection that whether the random environment changes qualitatively the behavior of the underlying random walk, at least at high temperature (\textit{i.e.} for small $\beta>0$), is an important topic in the study of polymer models. Of particular interest is whether small disorder (small $\beta>0$) can induce such change: if disorders do not change the qualitative behavior of the underlying random walk for small enough $\beta>0$, then we say the model is \textit{disorder irrelevant}, while if disorders change the model for arbitrarily small $\beta>0$, then we say the model is \textit{disorder relevant}.

It is well-known that the classic $(1+d)$-dimensional directed polymer model is disorder relevant for $d=1$ and disorder irrelevant for $d\geq3$, and the pinning model associated with a recurrent renewal process is disorder irrelevant for $\alpha\in(0,\frac{1}{2})$ and disorder relevant for $\alpha>\frac{1}{2}$ (\textit{cf.} \cite{C17,G11}). The $(1+2)$-dimensional directed polymer and the pinning model for 1-dimensional simple random walk are \textit{critical}, where the situation is more subtle: whether disorder is relevant or irrelevant depends on the finer structure of the model. Recently, it has been shown in \cite{BL17,BL18} that disorder is still relevant for the two critical models, but their properties are very different from that of related disorder relevant models, and the models are called \textit{marginally relevant} (by contrast, if a model is critical and disorder is irrelevant, then the model is called \textit{marginally irrelevant}).

Based on the same simple random walk $S$ on $\bbZ^d$, if the region $\Omega_N$ is widened from the time-space line $\bbN\times\{0\}$ to the full space $\bbN\times\bbZ^d$, then the model (\ref{dp}) changes from the pinning model to the classic directed polymer model. In particular, by the discussion above, when $d=1$, we expect to observe a transition of the model from marginal relevance to disorder relevance. A natural question then arises: how does the behavior of the model change as the spatial width of the tube environment is enlarged gradually?

\subsection{Directed polymers in tube environments and the main results} Motivated by the question above, we define a new variant of the classic directed polymer model, called \textit{directed polymer in tube environments}, by
\begin{equation}\label{dptube}
\frac{\dd\bP_{N,\beta}^{\omega}}{\dd\bP}(S):=\frac{1}{Z_{N,\beta}^{\omega}}\exp\left(\sum\limits_{n=1}^{N}(\beta\omega_{n,S_n}-\lambda(\beta))\mathbbm{1}_{\{|S_n|\leq RN^a\}}\right),
\end{equation}
where $|\cdot|$ denotes the Euclidean distance, and $R\geq0$ and $a\in[0,1]$ are parameters for the region of the random environment. 
It interpolates between the disordered pinning model and the classic directed polymer model, where the former corresponds to $R=0$ and the latter corresponds to $a=1$ and $R\geq1$.
The assumption that the width of the environment region grows polynomially fast is for the simplicity. If the region does not grow purely polynomially but with a slowly varying correction, then as will be discussed in Remark~\ref{r:region}, the model does not change much except at criticality.
We note that a similar setting where the environment region widens gradually has been considered in \cite{CD13} for the first passage percolation model.

Generally speaking, disorder relevant models are difficult to treat for fixed $\beta>0$. During the last ten years, an alternative approach has been developed, that is, studying the \textit{weak coupling limits} of the partition functions $Z_{N,\beta_N}^{\omega}$ by sending $\beta_N\downarrow0$ at a suitable rate as $N\to\infty$. Roughly speaking, if we tune down $\beta_N$ at just the right rate, then the effect of disorder will neither vanish nor blow up as $N\to\infty$, and the polymer measure has non-trivial disordered continuum limits (\textit{cf.} \cite{AKQ10,AKQ14,CSZ16,CSZ17a,CSZ17b}). However, if $\beta_N$ is too small, then it is as if disorder is not present at all, while if $\beta_N$ is too large, then the polymer measure will still undergo qualitative change, similar to the case $\beta_N\equiv\beta>0$. Notice that this picture is only valid for disorder (marginally) relevant models, resulting in some non-trivial limit $\cZ$ for $Z_{N,\beta_N}^{\omega}$, while for disorder irrelevant models, $Z_{N,\beta_N}^{\omega}$ trivially converges to $1$ in probability.

The first work to study the weak coupling limit was \cite{AKQ14}, where the authors considered the classic $(1+1)$-dimensional directed polymer model, which is disorder relevant. They proved that for $\beta_N=\hat{\beta}N^{-1/4}$ with any $\hat{\beta}>0$, $Z_{N,\beta_N}^{\omega}$ converges weakly to a non-trivial limit $\cZ_{\hat{\beta}}$. Later the authors of \cite{CSZ17b} proved that for several classes of marginally relevant polymer models, it is necessary to choose $\beta_N=\hat{\beta}/L(N)$ with $L(N)$ the expected overlap between two independent copies of the underlying process, and furthermore, there is a phase transition on this scale: if $\hat{\beta}<\hat{\beta}_c$, then the partition functions converge to a non-trivial limit $\bZ_{\hat{\beta}}$, while if $\hat{\beta}\geq\hat{\beta}_c$, then the weak limit is 0.

The above facts also illustrate an essential difference between disorder relevance and marginal relevance. 
In the disorder relevant regime, the suitable $\beta_N$ decays polynomially fast and the value of $\hat{\beta}$ is not crucial, \textit{i.e.}, there is no phase transition in $\hat{\beta}$. 
By contrast, in the marginally relevant regime, the suitable $\beta_N$  decays as a slowly varying function and a critical $\hat\beta_c$ splits the marginally relevant regime into two sub-regimes: $\hat\beta<\hat\beta_c$, the so-called {weak disorder regime}, and $\hat{\beta}\geq\hat\beta_c$, the {strong disorder regime}.

In this paper, we consider the weak coupling limits for directed polymers in tube environments in all dimensions $d\geq1$ and parameters $R\geq0$ and $a\in[0,1]$. Note that the case $a\in(0,1]$ but $R=0$ is the same as $a=0$ and $R<1$. We determine the weak coupling limits of the partition functions in all cases in the following theorem, where we write $Z_N=Z_{N,\beta_N}^{\omega}$ for simplicity. The proof will be given in Section~\ref{S:pf}. From the theorem we can observe a transition from disorder relevance to marginal relevance at $d=1, a=\frac{1}{2}$, and a transition from marginal relevance to disorder irrelevance at $d=2,a=0$.

\begin{theorem}\label{T1}
For directed polymers in tube environments \eqref{dptube}, the following convergence results hold.

\noindent\textbf{$\bullet$ Disorder relevant regime:} For $d=1$, $a\in[\frac{1}{2},1]$ and any $R>0$, let $\beta_{N}=\hat{\beta}N^{-\frac{1}{4}}$ for some $\hat{\beta}>0$. We have
\begin{equation}\label{conv1}
Z_{N}\overset{d}{\longrightarrow}\cZ_{\hat{\beta}}:=1+\begin{cases}\dis\sum\limits_{k=1}^{\infty}\frac{\hat{\beta}^{k}}{k!}\iint\limits_{([0,1]\times\bbR)^{\otimes k}}\psi_{k}((t_{1},x_{1}),\cdots,(t_{k},x_{k}))\prod\limits_{j=1}^{k}W({\rm d}t_{j},{\rm d}x_{j}),\quad\text{for}~a\in\bigg(\frac{1}{2},1\bigg],\\[5pt]
\dis\sum\limits_{k=1}^{\infty}\frac{\hat{\beta}^{k}}{k!}\iint\limits_{([0,1]\times[-R,R])^{\otimes k}}\psi_{k}((t_{1},x_{1}),\cdots,(t_{k},x_{k}))\prod\limits_{j=1}^{k}W({\rm d}t_{j},{\rm d}x_{j}),\quad\text{for}~a=\frac{1}{2}.
\end{cases}
\end{equation}
where $\overset{d}{\rightarrow}$ denotes weak convergence, $\psi_{k}((t_{1},x_{1}),\cdots,(t_{k},x_{k}))=\prod_{j=1}^{k}\sqrt{2}g_{t_{j}-t_{j-1}}(x_{j}-x_{j-1})$ on $0<t_1<\cdots<t_k<1$ and is symmetric in its arguments, where 
$
g_{t}(x)=\frac{1}{\sqrt{2\pi t}}e^{-x^{2}/2t}
$,
and $W(t,x)$ is the time-space white noise. Furthermore, $\bbE[(Z_N)^2]\underset{N\to\infty}{\longrightarrow}\bbE[(\cZ_{\hat{\beta}})^2]$.

\noindent\textbf{$\bullet$ Marginally relevant regime:} If\\
{\rm (i)} $d=1, a=0, R\geq0$ and $\beta_N=\hat{\beta}\sqrt{\frac{\pi}{(2\lfloor R\rfloor+1)\log N}}$; or\\[5pt]
{\rm (ii)} $d=1, a\in(0,\frac{1}{2}), R>0$, and $\beta_N=\hat{\beta}\sqrt{\frac{\pi}{2(1-2a)RN^a\log N}}$; or\\[5pt]
{\rm (iii)} $d=2, a\in(0,1], R>0$, and $\beta_N=\hat{\beta}\sqrt{\frac{\pi}{(2a\wedge1)\log N}}$,\\
then
\begin{equation}\label{conv2}
Z_{N}\overset{d}{\longrightarrow}\bZ_{\hat{\beta}}=\bZ_{\hat{\beta},d,a}:=\begin{cases}
\exp\left(\sigma_{\hat{\beta},d,a}W_{1}-\frac{\sigma_{\hat{\beta},d,a}^{2}}{2}\right),&\quad\mbox{if}~\hat{\beta}\in(0,1),\\
0,&\quad\mbox{if}~\hat{\beta}\geq1,
\end{cases}
\end{equation}
where $W_{1}$ is standard normal and
\begin{equation}\label{sigma}
\sigma_{\hat{\beta},d,a}^2=\begin{cases}
\log\frac{1-2a\hat{\beta}^2}{1-\hat{\beta}^2},\quad&\text{for}~d=1, a\in[0,\frac{1}{2}),\\
\log\frac{1}{1-\hat{\beta}^2},\quad&\text{for}~d=2,a\in(0,1].
\end{cases}
\end{equation}
Furthermore, $\bbE[(Z_N)^2]\overset{N\to\infty}{\longrightarrow}\bbE[(\bZ_{\hat{\beta}})^2]$.

\noindent\textbf{$\bullet$ Disorder irrelevant regime:} If {\rm (i)} $d=2, a=0, R\geq0$; or {\rm (ii)} $d\geq3, a\in[0,1], R\geq0$, then $Z_N$ converges to $1$ in $\bbP$-probability for any $\beta_N\downarrow0$.
	
\end{theorem}
\begin{remark}\label{R1}
Note that for case {\rm (i)} in the marginally relevant regime, we have that $\bP(|S_n|\leq R)=\sum_{x=-\lfloor R\rfloor}^{\lfloor R\rfloor}\bP(S_n=x)$, which falls within the scope of \cite{CSZ17b}. Indeed, for $|2m|, |2m+1|\leq R$, $\bP(S_{2k}=2m)\sim\bP(S_{2k}=0)$ and $\bP(S_{2k+1}=2m+1)\sim\bP(S_{2k+1}=1)$ as $k\to\infty$, which are asymptotically $(\pi k)^{-1/2}$ uniformly in $m$. Hence, let $S'$ be an i.i.d.\ copy of $S$. Then
\begin{equation*}
\sum_{n=1}^N\bP^{\otimes2}(S'_n=S_n)\sim\sum_{n=1}^{\lfloor\frac{N}{2}\rfloor}\sum\limits_{|x|\leq\lfloor R\rfloor}\left(\bP(S_{2n}=x)^2+\bP(S_{2n-1}=x)^2\right)\sim\frac{2\lfloor R\rfloor+1}{\pi}\log N,
\end{equation*}
which matches the setting for the critical pinning model in \cite{CSZ17b}.

Also note that for case {\rm (ii)} in the marginally relevant regime, the scale of $\beta_N$ changes smoothly in the parameter $a$ as $a$ varies from 0 to $\frac{1}{2}$, and then we observe a transition to the regime of disorder relevance at $a=\frac{1}{2}$.
When $a=0$, the model is essentially the same as a pinning model.
When $a$ is increased from $0$ to $\frac{1}{2}$, the variance of $\log\bZ_{\hat{\beta},1,a}$ decreases. In particular, if we let $a\to(\frac{1}{2})^-$ in \eqref{sigma}, then $\bZ_{\hat{\beta},1,a}\to1$ in $\bbP$-probability. This is an interesting but expected phenomenon, since the proper scaling for $\beta$ at $d=1,a=\frac{1}{2}$ is $N^{-1/4}$, while the scaling in case {\rm (ii)} in the marginally relevant regime has an extra $(\log N)^{-1/2}$ at $a=\frac{1}{2}$, which is too large for the disorder to have an effect in the limit.
\end{remark}

\subsection{Related studies and discussions}\label{S14}
With the convergence of the partition function $Z_N$ at hand, it is natural to ask whether the polymer measure admits a continuum limit if we also rescale time and space diffusively. We note that there have also been a lot of progress in studying the random field of partition functions as the starting point of the polymer, see \cite{CSZ20,CSZ21,HH12,LZ20}. Such questions arise naturally in the study of singular stochastic PDEs, but are less natural for our model of directed polymer in tube environments. In this subsection, we review some known results for the classic directed polymers and discuss their possible extensions to our model. Besides, we also briefly discuss the reason why our environment region is chosen as tubes.

\vspace{0.2cm}
\noindent\textbf{$\bullet$ Convergence of the polymer measure.} We first review the phase transition of the directed polymer. It is well-known that there is a critical $\beta_c\geq0$, depending on the dimension $d$, that splits the value of inverse temperature into the so-called \textit{weak disorder regime} ($\beta\in[0,\beta_c)\cup\{0\}$) and the \textit{strong disorder region} ($\beta>\beta_c$), see \textit{e.g.}\ \cite{C17}. In particular, $\beta_c=0$ for dimension $d=1,2$, which correspond to disorder relevance and marginal relevance, and $\beta_c>0$ for $d\geq3$, which corresponds to disorder irrelevance. Generally speaking, when $\beta<\beta_c$, the disorder is too weak to change qualitatively the long-scale behavior of the random walk under the polymer measure, while when $\beta>\beta_c$, the disorder becomes strong enough such that the long-scale behavior of the random walk is changed qualitatively under the polymer measure. We emphasize that the problems in the strong disorder regime are very difficult, where much less is known so far, compared to the weak disorder regime.

Let us start with spatial dimensions $d\geq3$. For $\beta\in(0,\beta_c)$, a series of results culminated in \cite{CV06}, where the authors showed that, after diffusive scaling, the polymer measure $\bP_{N,\beta}^\omega$ converges to the Wiener measure. To be precise, for any bounded continuous functional $F: (C[0,1], \bbR^d)\to\bbR$, as $N\to\infty$,
\begin{equation}\label{conv2wiener}
\bE_{N,\beta}^\omega\left[F(X_\cdot)\right]\to\bE_W\left[F(B_\cdot/\sqrt{d})\right]\quad\text{in}~\bbP\text{-probability,}
\end{equation}
where $(X_t)_{t\in[0,1]}$ is the linear interpolation of the diffusively rescaled simple random walk, and $(B_t)_{t\in[0,1]}$ is the standard $d$-dimensional Brownian motion under the Wiener measure $\bP_W$. Note that here $\beta$ is fixed and the weak coupling scaling is not needed.

For dimension $d=2$, although $\beta_c=0$ for the classic directed polymer, there actually exists the same type of phase transition from the weak disorder regime to the strong disorder region on the scale $\beta_{N}=\hat{\beta}\sqrt{\pi}/\sqrt{\log N}$ with critical point $\hat{\beta}_c=1$, as discovered in \cite{CSZ17b}. The weak disorder regime in dimension $d=2$ is essentially the same as that in higher dimensions. In fact, it is recently shown in \cite{G21} that, under the weak coupling $\beta_N=\hat{\beta}\sqrt{\pi}/\sqrt{\log N}$ with $\hat{\beta}\in(0,1)$, \eqref{conv2wiener} also holds.

In dimension $d=1$, by contrast, there is no phase transition on the intermediate disorder scale $\beta_N=\hat{\beta}N^{-1/4}$. Nevertheless, for $\hat{\beta}>0$, it is shown in \cite{AKQ14} that the polymer measure converges to the \textit{continuum $(1+1)$-dimensional directed polymer measure} (see \cite{AKQ14+}), which is formally defined by replacing the underlying random walk by a Brownian motion, the lattice random environment by a time-space white noise, and the summation in \eqref{dp} by integration. In particular, the authors showed that
\begin{equation}\label{1dconv}
\frac{\sqrt{N}}{2}\bP_{N,\beta_N}^\omega(S_{tN}=x\sqrt{N})\overset{d}{\longrightarrow}\frac{\cZ_{\sqrt{2}\hat{\beta}}(0,0,t,x)\cZ_{\sqrt{2}\hat{\beta}}(t,x,1,*)}{\cZ_{\sqrt{2}\hat{\beta}}(0,0,1,*)},
\end{equation}
where $\cZ_\beta(t,x):=\cZ_\beta(0,0,t,x)$ is the solution of the 1-dimensional stochastic heat equation (SHE) with multiplicative noise (\textit{e.g.}\ \cite{PS00})
\begin{equation}\label{SHE}
\partial_t\cZ=\frac12\Delta\cZ+\beta\cZ\dot{W},\quad\cZ_\beta(0,x)=\delta_0(x),
\end{equation}
where  $\dot{W}$ denotes time-space white noise, and
\begin{equation*}
\cZ_\beta(t,x,s,*)=\int_\bbR\cZ_\beta(t,x,s,y)\dd y.
\end{equation*}
We denote the \textit{point-to-point} partition function by
\begin{equation}\label{p2p}
Z_{N,\beta_N}^\omega(t,x,s,y):=\bE\left[\exp\left(\sum\limits_{n=tN+1}^{sN}(\beta_N\omega_{n,S_n}-\lambda(\beta_N))\right)\mathbbm{1}_{\{S_{sN}=y\sqrt{N}\}}\Bigg|\mathbbm{1}_{\{S_{tN}=x\sqrt{N}\}}\right],
\end{equation}
where $(tN,x\sqrt{N})$ and $(sN,y\sqrt{N})$ are interpreted as the closest time-space points in the even lattice $\{(n,z)\in\bbN\times\bbZ:n+z~\text{is even}\}$. In contrast,
\begin{equation}\label{p2pl}
Z_{N,\beta_N}^\omega(t,x,s,*):=\bE\left[\exp\left(\sum\limits_{n=tN+1}^{sN}(\beta_N\omega_{n,S_n}-\lambda(\beta_N))\right)\Bigg|\mathbbm{1}_{\{S_{tN}=x\sqrt{N}\}}\right],
\end{equation}
is called \textit{point-to-line} (or one-point) partition function. Note that in this article, what we actually study is $Z_{N,\beta_N}^\omega(0,0,1,*)$. In view of \eqref{1dconv}-\eqref{p2pl}, it is natural to think of the partition function of the classic directed polymers as a discretization of the solution of SHE. 
The key to proving \eqref{1dconv} is to observe that
\begin{equation*}
\frac{\sqrt{N}}{2}\bP_{N,\beta_N}^\omega(S_{tN}=x\sqrt{N})=\frac{\sqrt{N}}{2}\frac{Z_{N,\beta_N}^\omega(0,0,t,x)Z_{N,\beta_N}^\omega(t,x,1,*)}{Z_{N,\beta_N}^\omega(0,0,1,*)},
\end{equation*}
and
\begin{equation*}
	\Big\{\frac{\sqrt{N}}{2}Z_{N,\beta_N}^\omega(t,x,s,y):0\leq t<s\leq 1,x,y\in\bbR\Big\}\overset{d}{\longrightarrow}
	\left\{\cZ_{\sqrt{2}\hat\beta}(t,x,s,y):0\leq t<s\leq 1,x,y\in\bbR\right\},
\end{equation*}
where the scaling factor $\sqrt{N}/2$ is due to the extra transition kernel to the end point and the period of the underlying simple random walk.
This suggests how to extend our model for the disorder relevant case to convergence of the polymer measure. We believe that in the weak coupling limit as in Theorem \ref{T1}, for $d=1,a\in(\frac{1}{2},1]$, we should obtain the same limit as the continuum directed polymer from \cite{AKQ14+}, because the polymer remains diffusive. For $d=1,a\in[0,\frac{1}{2})$ and $d=2,a\in(0,1]$, \eqref{conv2wiener} should also hold by adapting the method in \cite{G21}. The case $d=1,a=\frac{1}{2}$ is a bit subtle, and we expect that $\bP_{N,\beta_N}^\omega$ to converge to a continuum directed polymer measure defined through a white noise restricted to a time-space tube. Finally, for dimension $d\geq3$, disorder is irrelevant, and the same result as in \cite{CY06} should hold. Proof should be similar, using a second moment calculations. 

The most difficult case will be the scaling limit of the polymer measure in the marginally relevant regime at the critical point. As discussed in \cite{CSZ21}, for the classic $(1+2)$-dimensional directed polymer model, there is in fact a critical window $\hat{\beta}_c(1+\frac{\theta}{\log N})$ around the critical point $\hat{\beta}_c$ in $\beta_N=\hat{\beta}\sqrt{\pi}/\sqrt{\log N}$, which interpolates between the weak disorder and strong disorder regime, similar to the disorder scale $\beta_N=\hat{\beta}N^{-1/4}$ in dimension $d=1$. However, the partition function converges to $0$ in this critical window, so there is no analogue of \cite{CSZ20}. Nevertheless, the recent breakthrough \cite{CSZ21} showed that when the partition function is averaged over the starting point of the polymer, it has a scaling limit. This raises hope that convergence of polymer measure can also be obtained, at least if the starting point of the polymer is chosen uniformly from a ball instead of starting from a fixed point.

\vspace{0.2cm}
\noindent\textbf{$\bullet$ The choices of the environment region.} An alternative choice of the environment region is $\Omega_N:=\{(n,x): 1\leq n\leq N, |x|\leq Rn^a\}$. In that case, the model can be called \textit{directed polymers in cone environments}. However, it turns out that the model in cone environments has no essential difference compared to the model in tube environments and the former one is less convenient to handle. Thus, we will focus on the model in tube environments in this paper, which is enough to illustrate transitions from the pinning model to the classic directed polymer model. We will come back to this issue in Remark \ref{r:cone}.

\section{Proof of Theorem~\ref{T1}}\label{S:pf}
The general strategies of proofs for weak coupling limits have been elaborated in \cite{AKQ14,CSZ17a,CSZ17b}.
We are thus able to provide, based on existing results, a short proof of Theorem~\ref{T1} in Subsection~\ref{S:pfd}, without going through the whole picture. Nevertheless, to illustrate the key ideas behind them, we will first sketch some heuristics in Subsection~\ref{S:pfs}.
\subsection{Heuristics}\label{S:pfs}
We start with a standard polynomial chaos expansion for the partition function $Z_{N}$. For every $N$, we define
\begin{equation}\label{xi}
\xi_{n,x}:=\xi_{n,x}^{(N)}=\beta_{N}^{-1}(e^{\beta_{N}\omega_{n,x}-\lambda(\beta_{N})}-1).
\end{equation}

It is easy to check that $\bbE[\xi_{n,x}]=0$ and $\bbV\!\mbox{ar}(\xi_{n,x})=1+o_N(1)$ as $\beta_N\downarrow0$.
Using $e^{x\mathbbm{1}_{A}}=1+(e^{x}-1)\mathbbm{1}_{A}$ in the expression (\ref{Z}) of $Z_N$, we can rewrite
\begin{equation}\label{poly}
\begin{split}
&Z_{N}=\bE\left[\prod_{n=1}^{N}\Big(1+\beta_{N}\xi_{n,S_{n}}\mathbbm{1}_{\{|S_{n}|\leq RN^{a}\}}\Big)\right]\\
=&1+\sum\limits_{k=1}^{N}\beta_{N}^{k}\sum\limits_{1\leq n_{1}<\cdots<n_{k}\leq N\atop(x_{1},\cdots,x_{k})\in B(RN^{a})^{\otimes k}}\prod_{j=1}^{k} p_{n_j-n_{j-1}}(x_j-x_{j-1}) 
\xi_{n_{j},x_{j}}:=1+\sum_{k=1}^{N}\beta_N^k Z_{N,k},
\end{split}
\end{equation}
where $B(r)=\{x\in\bbZ^d:|x|\leq r\}$, $p_{n}(x)=\bP(S_{n}=x)$, and by convention $n_{0}=0$ and $x_{0}$ is the origin.

Since for $k\neq l$, there must be some $\xi_{n,x}$ appears exactly once in the product of $Z_{N,k}$ and $Z_{N,l}$, by the independence of $\xi_{n,x}$ we have $\bbE[Z_{N,k}Z_{N,l}]=0$ and thus $(Z_{N,k})_{1\leq k\leq N}$ are orthogonal in the $L_2$ space. Therefore, to establish the weak convergence of $Z_N$, it suffices to approximate $\beta_N^k Z_{N,k}$ in $L_2$ for $k\leq K$ given that $\sum_{k=K}^{\infty}\beta_N^k Z_{N,k}$ vanishes uniformly in $N$ as $K\to\infty$.
For any $k$, a second moment calculation (also see (4.3) in \cite{CSZ17b}) reveals that
\begin{equation}\label{varZnk}
\begin{split}
&\var(Z_{N,k})=\bbE\left[\xi_{1,0}^2\right]^k\sum\limits_{1\leq n_{1}<\cdots<n_{k}\leq N\atop(x_{1},\cdots,x_{k})\in B(RN^{a})^{\otimes k}}\prod_{j=1}^{k}p_{n_j-n_{j-1}}^2(x_j-x_{j-1})\\
\leq&\bbE\left[\xi_{1,0}^2\right]^k\sum\limits_{1\leq n_{1}-n_0<\cdots<n_{k}-n_{k-1}\leq N\atop\forall j, |x_j-x_{j-1}|\leq 2RN^a}\prod_{j=1}^{k}p_{n_j-n_{j-1}}^2(x_j-x_{j-1})\leq (c_d\var(Z_{1,k}))^k.
\end{split}
\end{equation}
Hence, to achieve a non-trivial limit, it is natural to choose $\beta_N$ of order $\var(Z_{N,1})^{-1/2}$ in view of \eqref{poly}-\eqref{varZnk}.

To calculate the variance of $Z_{N,1}$, we can decompose
\begin{equation*}
\var(Z_{N,1})\big/\bbE\left[\xi_{1,0}^2\right]=\sum_{n=1}^{N}\sum_{ |x|\leq RN^a} p_n^2(x)= \Bigg(\sum_{n=1}^{N^{1\wedge 2a}}\sum_{|x|\leq RN^a} + \sum_{n=N^{1\wedge 2a}}^{N}\sum_{|x|\leq RN^a}\Bigg)p_n^2(x) :=(\sigma_1^2+\sigma_2^2),
\end{equation*}
where the contribution to $\sigma_1^2$ (resp.\ $\sigma_2^2$) comes from the short (resp.\ long) time interval. A simple comparison as given in Lemma~\ref{L:I_N} below shows that only when $d=1,a\in(0,\frac{1}{2})$, the variance is dominated by $\sigma^2_2$, which suggests that in that case the model should behave similarly to the pinning model associated to the 1-dimensional simple random walk, as the environment region is thin (sub-diffusive) for large time scale. On the other hand, when $d=1,a\in[\frac{1}{2},1]$ or $d\geq2,a\in(0,1]$, the term $\sigma_1^2$ dominates, and it suggests that the model should be comparable to the classic directed polymer model. By \cite{AKQ14}, the classic $(1+1)$-dimensional directed polymer model is disorder relevant, while the $(1+2)$-dimensional directed polymer model and the pinning model associated to the 1-dimensional simple random walk are marginally relevant by \cite{CSZ17b}. For $d\geq3$, since $\var(Z_{N,1})$ is finite, the second moment of $Z_{N,\beta}^{\omega}$ is uniformly bounded for small enough $\beta>0$, and hence the model is disorder irrelevant.

In the disorder relevant and marginally relevant regimes, the limiting objects turn out to be some limits of Wiener chaos expansions.
Indeed, by a version of the Lindeberg principle \cite[Theorem~2.6]{CSZ17a} or \cite[Theorem~4.2]{CSZ17b}, $\xi_{n,x}$ in (\ref{poly}) can be replaced by independent standard normal random variables $W_{n,x}$ without affecting the limiting distribution.
On the diffusive scale, the local limit theorem allows us to approximate the weight function $p_n(x)$ by the Gaussian kernel.
Therefore, $Z_{N,k}$ can be written asymptotically as a multiple integral with respect to white noise with Gaussian weight over diffusively rescaled time and space.
With $\beta_N$ on the right scale as described above, $Z_N$ can be seen to converge to (\ref{conv1}) or (\ref{conv2}), once the covariance structure of $Z_{N,k}$ is determined.

\subsection{Proof}\label{S:pfd}
Following the heuristics above, the first crucial step for the proof is to estimate $\var(Z_{N,1})$, which is given by
\begin{equation*}
\var(Z_{N,1})=\sum_{1\leq n\leq N}\sum_{x\in B(RN^a)} \bP(S_n=x)^2 \bbE[\xi_{n,x}^2]
=\sum_{1\leq n\leq N}\bP^{\bigotimes2}\big(S_{n}=S_{n}^{'},|S_{n}|\leq RN^{a}\big)\big(1+o_N(1)\big),
\end{equation*}
where in the first equality we used \eqref{Evarw} and the independence of $\xi_{n,x}$, and in the second equality that $S'$ is an independent copy of $S$.
The following lemma determines its asymptotics. 
\begin{lemma}\label{L:I_N}
Let $a\in[0,1]$ and $R\geq0$.
Denote the expected intersection times within the ball $B(RN^a)$ of two independent simple random walks by
\begin{equation}\label{i_n}
I_{N}:=I_{N}^{a,R}:=\sum_{n=1}^{N}\bP^{\bigotimes2}(S_{n}=S_{n}^{'},|S_{n}|\leq RN^{a}).
\end{equation}
Then as $N\to\infty$, the following asymptotics hold:
\begin{equation*}
I_{N}\sim\begin{cases}
\dis C_{1,a,R}\sqrt{N},\quad&\mbox{for}~d=1,a\in[\frac{1}{2},1],R>0,\\[8pt]
\dis\frac{2(1-2a)}{\pi}RN^{a}\log N,\quad&\mbox{for}~d=1,a\in(0,\frac{1}{2}),R>0,\\[8pt]
\dis\frac{2\lfloor R\rfloor+1}{\pi}\log N,\quad&\mbox{for}~d=1, a=0, R\geq0,\\[8pt]
\dis\frac{1\wedge 2a}{\pi}\log N,\quad&\mbox{for}~d=2,a\in(0,1], R>0,\\[8pt]
\dis C_{d,a,R},\quad&\mbox{for}~d=2, a=0,R\geq0~\mbox{or}~d\geq3, a\in[0,1], R\geq0,
\end{cases}
\end{equation*}
where
\begin{equation*}
C_{d,a,R}=\begin{cases}
\dis\frac{2}{\sqrt{\pi}},\quad&\text{for}~d=1, a\in(\frac{1}{2},1],\\[5pt]
\dis\frac{1}{\pi}\int_{0}^{1}\frac{1}{\sqrt{t}}\int_{|x|\leq R/\sqrt{t}}e^{-|x|^{2}}\di x\di t,\quad&\text{for}~d=1, a=\frac{1}{2},\\[5pt]
\dis\sum\limits_{n=1}^{\infty}\bP^{\bigotimes2}(S_n=S_n', |S_n|\leq R),\quad&\text{for}~d\geq2,a=0,\\[5pt]
\dis\sum\limits_{n=1}^{\infty}\bP(S_{2n}=0),\quad&\text{for}~d\geq3,a\in(0,1].
\end{cases}
\end{equation*}
	
\end{lemma}
\begin{proof}
Note that
\begin{equation}\label{summand}
\bP^{\bigotimes2}(S_{n}=S_{n}^{'},|S_{n}|\leq RN^{a})=\sum\limits_{|x|\leq RN^{a}\wedge n}\bP(S_{n}=x)^{2}.
\end{equation}

Recall that the case $a>0, R=0$ is the same as $a=0, R<1$, and the case $d=1, a=0$ has actually been proved in \cite{CSZ17b}, see Remark \ref{R1}. By the local limit theorem, $\bP(S_n=x)\leq C_d n^{-d/2}$ uniformly for all $n$ and $x$. Hence, For $d\geq2, a=0$, $\eqref{summand}\leq C_{d,R}n^{-d}$, which is summable, and thus $C_{d,0,R}=\sum_{n=1}^{\infty}\bP(S_n=S_n', |S_n|\leq R)$. For $d\geq3$, $a\in(0,1]$, $\eqref{summand}\leq\bP(S_{2n}=0)$, which is also summable since the simple random walk in dimension $d\geq3$ is transient. Hence, by the dominated convergence theorem, $C_{d,a,R}=\sum_{n=1}^\infty\bP(S_{2n}=0)$.
	
We then focus on $d=1,2$ and $a\in(0,1]$. If $x$ can be visited at time $n$, then we write $x\leftrightarrow n$. By a finer local limit theorem \cite[Theorem~1.2.1]{L13},
\begin{equation}\label{I_N}
I_{N}=\sum_{\cA_N}\left(2\left(\frac{d}{2\pi n}\right)^{\frac{d}{2}}e^{-\frac{d|x|^{2}}{2n}}+E(n,x)\right)^{2},
\end{equation}
where $E(n,x)=O(n^{-1-d/2}\wedge|x|^{-2}n^{-d/2})$ and $\sum_{\cA_N}:=\sum_{n=1}^{N}\sum_{|x|\leq RN^{a}, x\leftrightarrow n}$. Note that
\begin{equation*}
\begin{split}
I_N&=4\left(\frac{d}{2\pi}\right)^{d}\sum\limits_{\cA_N}\frac{1}{n^{d}}e^{-\frac{d|x|^{2}}{n}}+4\left(\frac{d}{2\pi}\right)^{\frac{d}{2}}\sum\limits_{\cA_N}\frac{1}{n^{\frac{d}{2}}}e^{-\frac{d|x|^{2}}{2n}}E(n,x)+\sum\limits_{\cA_N}E(n,x)^{2}\\
&:=J_{N}^{(1)}+J_{N}^{(2)}+J_{N}^{(3)}.
\end{split}
\end{equation*}
We then estimate the three terms above.

\noindent\textbf{Term $J_{N}^{(1)}$.} Note that $S$ has period $2$. By a Riemann sum approximation, we have
\begin{equation}\label{J1}
J_{N}^{(1)}\sim 2\left(\frac{d}{2\pi}\right)^{d}\int_{1}^{N}\frac{1}{n^{d}}\int_{|x|\leq RN^{a}}e^{-\frac{d|x|^{2}}{n}}\di x\di n
=\frac{2d^{\frac{d}{2}}}{(2\pi)^{d}}\int_{1}^{N}\frac{1}{n^{\frac{d}{2}}}\int_{|x|\leq RN^{a}\sqrt{\frac{d}{n}}}e^{-|x|^{2}}\di x\di n.
\end{equation}
We need to consider (i) $a>\frac{1}{2}$, (ii) $a=\frac{1}{2}$, and (iii) $a<\frac{1}{2}$ separately.

\noindent\textbf{Case $\bf a>\frac{1}{2}$.} The inner integral converges to $\int_{\bbR^{d}}\exp(-|x|^2)dx$ for any $R>0$, and hence
\begin{equation*}
J_{N}^{(1)}\sim\begin{cases}
\frac{2}{\sqrt{\pi}}\sqrt{N},&\quad\mbox{for}~d=1,\\[5pt]
\frac{1}{\pi}\log N,&\quad\mbox{for}~d=2.
\end{cases}
\end{equation*}

\noindent\textbf{Case $\bf a=\frac{1}{2}$.} We split the domain $[1,N]$ of $n$ into two parts: (1) $[1,\epsilon N]$ and (2) $[\epsilon N,N]$, where $\epsilon$ is small.
For the first part, there exists $\delta_{\epsilon,d}\overset{\epsilon\to0}{\longrightarrow}0$, such that
\begin{equation*}
\int_{1}^{\epsilon N}\frac{1}{n^{\frac{d}{2}}}\int_{|x|\leq R\sqrt{\frac{dN}{n}}}e^{-|x|^{2}}{\rm d}x{\rm d}n\in\begin{cases}
\big[2(1-\delta_{\epsilon,1})\sqrt{\epsilon\pi N},2\sqrt{\epsilon\pi N}\big],&\quad\mbox{for}~d=1,\\[5pt] 
\big[\pi(1-\delta_{\epsilon,2})(\log N+\log\epsilon),\pi\log N\big],&\quad\mbox{for}~d=2,
\end{cases}
\end{equation*}
where we use the fact that $\int_{\bbR^d}e^{-|x|^2}\di x=\pi^{d/2}$ and the integral domain for $x$ is at least $\{|x|\leq R\sqrt{d/\epsilon}\}$, which approximates $\bbR^d$. For the second part, by a change of variable $n=tN$,
\begin{equation*}
\int_{\epsilon N}^{N}\frac{1}{n^{\frac{d}{2}}}\int_{|x|\leq R\sqrt{\frac{dN}{n}}}e^{-|x|^{2}}{\rm d}x{\rm d}n=N^{1-\frac{d}{2}}\int_{\epsilon}^{1}\frac{1}{t^{\frac{d}{2}}}\int_{|x|\leq R\sqrt{\frac{d}{t}}}e^{-|x|^{2}}{\rm d}x{\rm d}t.
\end{equation*}
For $d=1$, since $t^{-1/2}$ is integrable on $[0,1]$, we can let $\epsilon\to0$ and the above integral is finite. For $d=2$, the above integral is finite for any fixed $\epsilon$. Therefore,
\begin{equation*}
J_{N}^{(1)}\sim\begin{cases}
\big(\frac{1}{\pi}\int_{0}^{1}\frac{1}{\sqrt{t}}\int_{|x|\leq R/\sqrt{t}}e^{-|x|^{2}}{\rm d}x{\rm d}t\big)\sqrt{N},&\quad\mbox{for}~d=1,\\[5pt]
\frac{1}{\pi}\log N,&\quad\mbox{for}~d=2.
\end{cases}
\end{equation*}

\noindent\textbf{Case $\bf a<\frac{1}{2}$.} For $d=1$, we split the range of $[1,N]$ into two parts: (1) $[1, KN^{2a}]$ and (ii) $[KN^{2a},N]$, where $K$ is large.
For the first part,
\begin{equation*}
\int_{1}^{KN^{2a}}\frac{1}{\sqrt{n}}\int_{|x|\leq RN^{a}/\sqrt{n}}e^{-|x|^{2}}{\rm d}x{\rm d}n\leq C_{K,R}N^{a}.
\end{equation*}
For the second part, note that $e^{x}\overset{x\to0}{\longrightarrow}1$. Let $N\to\infty$ and then let $K\to\infty$. We obtain
\begin{equation*}
\int_{KN^{2a}}^{N}\frac{1}{\sqrt{n}}\int_{|x|\leq RN^{a}/\sqrt{n}}e^{-|x|^{2}}{\rm d}x{\rm d}n\sim 2RN^{a}\int_{KN^{2a}}^{N}\frac{1}{n}dn\sim2(1-2a)RN^{a}\log N.
\end{equation*}
Hence, recall \eqref{J1} and $J_{N}^{(1)}\sim\frac{2(1-2a)}{\pi}RN^{a}\log N$ for $d=1, a\in(0,\frac{1}{2})$.

For $d=2$, we split the range of $[1,N]$ into three parts: (1) $[1,\epsilon N^{2a}]$, (2) $[\epsilon N^{2a},KN^{2a}]$ and (3) $[KN^{2a},N]$, where $K$ is large and $\epsilon$ is small.
For the first part, there exists $\delta_{\epsilon}\overset{\epsilon\to0}{\longrightarrow}0$, such that
\begin{equation*}
\int_{1}^{\epsilon N^{2a}}\frac{1}{n}\int_{|x|\leq RN^{a}\sqrt{\frac{2}{n}}}e^{-|x|^{2}}{\rm d}x{\rm d}n\in[(1-\delta_{\epsilon})\pi(2a\log N+\log\epsilon),2a\pi\log N].
\end{equation*}
For the second part,
\begin{equation*}
\int_{\epsilon N^{2a}}^{KN^{2a}}\frac{1}{n}\int_{|x|\leq RN^{a}\sqrt{\frac{2}{n}}}e^{-|x|^{2}}{\rm d}x{\rm d}n\xlongequal{ n=N^{2a}t}{}\int_{\epsilon}^{K}\frac{1}{t}\int_{|x|\leq R\sqrt{\frac{2}{t}}}e^{-|x|^{2}}{\rm d}x{\rm d}t,
\end{equation*}
which is finite for any fixed $\epsilon,K$.
For the third part, note that $e^{x}\overset{x\to0}{\longrightarrow}1$. Let $N\to\infty$ and then let $K\to\infty$. We obtain
\begin{equation*}
\int_{KN^{2a}}^{N}\frac{1}{n}\int_{|x|\leq RN^{a}\sqrt{\frac{2}{n}}}e^{-|x|^{2}}{\rm d}x{\rm d}n\sim CR^{2}N^{2a}\int_{KN^{2a}}^{N}\frac{1}{n^{2}}{\rm d}n\sim\frac{CR^{2}}{K}.
\end{equation*}
Therefore, recall \eqref{J1} and $J_{N}^{(1)}\sim\frac{2a}{\pi}\log N$ for $d=2,a\in(0,\frac{1}{2})$.

\noindent\textbf{Term $J_{N}^{(2)}$.} Note that $E(n,x)=o(n^{\frac{d}{2}})$, so $J_{N}^{(2)}$ is negligible compared to $J_{N}^{(1)}$.

\noindent\textbf{Term $J_{N}^{(3)}$.} 
Recall that $E(n,x)=O(n^{-1-d/2}\wedge|x|^{-2}n^{-d/2})$.
We have
\begin{equation*}
J_N^{(3)}\leq\sum\limits_{n=1}^{N}\sum\limits_{|x|\leq n}n^{-(d+2)}\leq C\sum\limits_{n=1}^{\infty}n^{-2}<\infty. 
\end{equation*}

Combining all estimates above, the lemma for $d=1,2$ follows.
\end{proof}


We are now able to complete the proof of Theorem~\ref{T1}.
We will separately treat the cases of disorder irrelevance, disorder relevance, and marginal relevance. \vspace{10pt}

\noindent {\bf $\bullet$ Disorder irrelevant regime.} We prove the result by a second moment estimate. Recall $Z_{N,k}$ from \eqref{poly}. Using $\var(Z_{N,k})\leq(c\var(Z_{N,1}))^k$ (see \eqref{varZnk} and Lemma \ref{L:I_N}), we have that 
\begin{equation}\label{2nd}
\bbE|Z_N-1|^2=\sum\limits_{k=1}^{N}\beta_N^{2k}\var(Z_{N,k})\leq\sum\limits_{k=1}^{\infty}(C\beta_N)^{2k}=\frac{C\beta_N^2}{1-C\beta_N^2}\to0,\quad\text{for arbitrary}~\beta_N\downarrow0.
\end{equation}
\vspace{5pt}

\noindent{\bf $\bullet$ Disorder relevant regime.} For disorder relevant systems, \cite[Theorem~2.3]{CSZ17a} states a general convergence criterion. We will prove this case by verifying the conditions of \cite[Theorem~2.3]{CSZ17a}. Let us introduce a time-space lattice $\bT_N$ and its associated rescaled lattice $\bbT_N$ by
\begin{equation}
\bT_N:=\{(n,x)\in\bbN\times\bbZ :n\leftrightarrow x\},\quad\bbT_{N}:=\{(t,y)\in [0,1]\times\bbR:(Nt,\sqrt{N}y)\in\bT_N\}.
\end{equation}

Let $(\zeta_z)_{z\in\bbT_N}$ be such that $\zeta_z=\xi_{n,x}$, where $z:=z^{(N)}=(t,y)\in\bbT_N$ with $(Nt,\sqrt{N}y)=(n,x)$. Then by \eqref{xi},
\begin{equation}\label{zeta}
\bbE[\zeta_z]=0\qquad{\rm and}\qquad \var(\zeta_z)=1+o_N(1),\quad{\rm as~}N\to\infty.
\end{equation}

Introduce a symmetric function $\psi_{N,k}$ of $(z_1,\ldots,z_k)\in\bbT_N^{\otimes k}$. It vanishes if $z_i=z_j$ for any $i\neq j$, and for distinct $z_1=(t_1,y_1),\ldots,z_k=(t_k,y_k)$ with $0<t_1<\ldots<t_k\leq 1$,
\begin{equation*}
\psi_{N,k}((t_1,y_1),\ldots,(t_k,y_k))=\prod_{j=1}^k N^{-1/4} p_{(t_j-t_{j-1})N}((y_j-y_{j-1})\sqrt{N})\mathbbm{1}_{\{|y_j|\leq RN^{a-1/2}\}}.
\end{equation*}
Moreover, we define $\psi_{N,k}((t_1,y_1),\ldots,(t_k,y_k))=0$, if $(Nt_i,\sqrt{N}y_i)=(n_i,x_i)\in\bbN\times\bbZ$ ($i=1,\ldots,k$) but $\|x_i\|_1>n_i$ for some $i$.
Then we can write
\begin{equation*}
Z_{N,k}=1+\sum_{k=1}^{N}\frac{\hat{\beta}^k}{k!}\sum_{(z_1,\ldots,z_k)\in\bbT_N^{\otimes k}}\psi_{N,k}(z_1,\ldots,z_k) \prod_{j=1}^{k}\zeta_{z_j}.
\end{equation*}

Let $\mathcal{C}_N$ be the tessellation of $[0,1]\times\bbR$ indexed by $(t,y)\in[0,1]\times\bbR$ with $(Nt,\sqrt{N}y)=(n,x)\in\bbN\times\bbZ$ and $n+x$ is even, $\cC_N(z)=(t-\frac{1}{2N},t+\frac{1}{2N}]\times(y-\frac{1}{\sqrt{N}},y+\frac{1}{\sqrt{N}}]$, and $v_N:=|\cC_N(z)|=2N^{-3/2}$.
We extend the domain of $\psi_{N,k}$ to $\big([0,1]\times\bbR\big)^{\otimes k}$ by defining $\psi_{N,k}(\tilde z):=\psi_{N,k}(z)$ for $\tilde z\in\cC_N(z)$.

We then check the three conditions of \cite[Theorem~2.3]{CSZ17a}. Note that the condition (i) is directly satisfied by \eqref{zeta}. To verify the condition (ii), we need to show that $\|v_N^{-k/2}\psi_{N,k}-\psi_k\|_2\overset{N\to\infty}{\longrightarrow}0$. First, for distinct $z_1=(t_1,y_1),\cdots, z_k=(t_k,y_k)$, define its lattice approximation $(t_{N,i},y_{N,i})=(N^{-1} n_{N,i}, N^{-1/2}x_{N,i})$ such that $z_{N,i}=(n_{N,i}, x_{N,i})\in\bbN\times\bbZ$ and $z_i\in\mathcal{C}_N(z_{N,i})$. By the local limit theorem,
\begin{equation*}
\begin{split}
v_N^{-k/2}\psi_{N,k}(z_1,\ldots,z_k)&=2^{-k/2}N^{3k/4}\prod\limits_{j=1}^{k}N^{-k/4}p_{(t_{N,j}-t_{N,j-1})N}((y_{N,j}-y_{N,j-1})\sqrt{N})\mathbbm{1}_{\{|y_j|\leq RN^{a-1/2}\}}\\
\hspace{0.5cm}\overset{N\to\infty}{\longrightarrow}\psi_k(z_1,\ldots,z_k):&=\begin{cases}\prod\limits_{j=1}^{k}\sqrt{2}g_{t_j-t_{j-1}}(y_j-y_{j-1})\mathbbm{1}_{\{|y_j|\leq R\}},&\quad\text{if}~a=\frac{1}{2},\\
\prod\limits_{j=1}^{k}\sqrt{2}g_{t_j-t_{j-1}}(y_j-y_{j-1}),&\quad\text{if}~a\in(\frac{1}{2},1].
\end{cases}
\end{split}
\end{equation*}

Next, by the local large deviation of $1$-dimensional simple random walk (\textit{cf.} \cite[Theorem 3]{S67}), we have that $p_t(y)\leq Ct^{-1/2}\exp(-cy^2/t)$ uniformly in $t$ and $y$ for some constant $C, c>0$. Note that
\begin{equation*}
\int\cdots\int_{\bbR^k}\prod\limits_{j=1}^k\exp\left(-c\frac{(y_j-y_{j-1})^2}{t_j-t_{j-1}}\right)\di y_1\cdots\di y_k=\prod\limits_{j=1}^{k}C\sqrt{t_j-t_{j-1}}.
\end{equation*}
Then a Riemann sum approximation yields that
\begin{equation*}
\begin{split}
&\|v_N^{-k/2}\psi_{N,k}\|_2^2\\
\leq&k!C^k\int\cdots\int_{0<t_1<\cdots<t_k<1}\prod_{j=1}^k\frac{1}{t_j-t_{j-1}}\int\cdots\int_{\bbR^k}\prod\limits_{j=1}^k\exp\left(-c\frac{(y_j-y_{j-1})^2}{t_j-t_{j-1}}\right)\di y_1\cdots\di y_k\di t_1\cdots\di t_k\\
\leq&k!C^k\int\cdots\int_{0<t_1<\cdots<t_k<1}\prod_{j=1}^k\frac{1}{\sqrt{t_j-t_{j-1}}}\di t_1\cdots\di t_k\leq k!C^k k^{-ck},
\end{split}
\end{equation*}
where the last inequality is due to \cite[Lemma B.2]{CSZ17a}. Hence, $\lim_{N\to\infty}\|v_N^{-k/2}\psi_{N,k}\|_2=\|\psi_k\|_2$ by the dominated convergence theorem.

Now suppose that $f_N\overset{N\to\infty}{\longrightarrow}f$ almost surely and $\|f_N\|_p\overset{N\to\infty}{\longrightarrow}\|f\|_p$. Consider the function $2^p|f_N|^p+2^p|f|^p-|f_N-f|^p$, which is nonnegative since $|f_N-f|^p\leq(2\max\{|f_N|,|f|\})^p\leq2^p(|f_N|^p+|f|^p)$. By Fatou's Lemma,
\begin{equation*}
2^{p+1}\|f\|_p^p+\varliminf\limits_{N\to\infty}-\|f_N-f\|_p^p=\varliminf\limits_{N\to\infty}(2^p\|f_N\|_p^p+2^p\|f\|_p^p-\|f_N-f\|_p^p)\geq2^{p+1}\|f\|_p^p,
\end{equation*}
which implies $\varlimsup_{N\to\infty}\|f_N-f\|_p^p\leq0$. Hence, $\|v_N^{-k/2}\psi_{N,k}-\psi_k\|_2\overset{N\to\infty}{\longrightarrow}0$.

It remains to check the condition (iii), which follows immediately from $\sum_{k=0}^\infty C^k k^{-ck}<+\infty$. Then the proof is completed by a direct application of \cite[Theorem~2.3]{CSZ17a}.
\vspace{10pt}

\noindent{\bf $\bullet$ Marginally relevant regime.} We prove the results in this regime by showing that $Z_N$ is arbitrarily close to the partition function of the pinning or directed polymer model whose weak coupling limits are already known. 
Recall that $\beta_N=\hat{\beta}/\sqrt{I_N}$. 
We will first treat the case $\hat{\beta}\in(0,1)$ and then $\hat{\beta}\geq1$.

\vspace{0.15cm}
\noindent{\bf (i) Case $\bf\hat{\beta}\in(0,1)$.} We perform a second moment estimate similar to \eqref{2nd}. For any positive integer $K$, we have that
\begin{equation*}
\sum\limits_{k=K}^{\infty}\beta_N^{2k}\var(Z_{N,k})\leq\sum\limits_{k=K}^{\infty}((1+o_N(1))\hat{\beta})^{2k}\overset{K\to\infty}{\longrightarrow}0.
\end{equation*}

Hence, as described in \cite[Section 4]{CSZ17b}, it suffices to show that $1+\sum_{k=1}^K \beta_N^k Z_{N,k}$ converges in distribution to some random variable $\bZ_{\hat{\beta},K}$ as $N\to\infty$ and then $\bZ_{\hat{\beta},K}$ converges in $L_2$ to the desired limit $\bZ_{\hat{\beta}}$ in \eqref{conv2} as $K\to\infty$. In particular, we only need to tackle $Z_{N,k}$ for finitely many $k=1,\cdots,K$. We actually show that as $N\to\infty$, $Z_{N,k}$ is arbitrarily close in $L_2$ to the \textit{spatially constraint-free} $k$-th term of some polynomial chaos expansion, which is
\begin{equation}\label{hatZ}
\hat{Z}_{N,k}:=\begin{cases}
\sum\limits_{1\leq n_1<\cdots<n_k\leq N\atop\forall j, n_j-n_{j-1}\geq N^{2a}}\prod\limits_{j=1}^{k}p_{n_j-n_{j-1}}(0)\xi_{n_j,0},&\quad\mbox{for}~d=1, a\in(0,\frac{1}{2}),\\
\sum\limits_{1\leq n_1<\cdots<n_k\leq N\atop (x_1,\cdots,x_k)\in\bbZ^k}\prod\limits_{j=1}^{k}p_{n_j-n_{j-1}}(x_j-x_{j-1})\xi_{n_j,x_j},&\quad\mbox{for}~d=2, a\in(0,1].
\end{cases}
\end{equation}

Note that for $d=2$, $\hat{Z}_{N,k}$ is exactly the $k$-th term of the polynomial chaos expansion for the partition function of the classic $(1+2)$-dimensional directed polymer model. For $d=1$, $\hat{Z}_{N,k}$ is related to the partition function of a pinning model associated to a $1$-dimensional simple random walk, with the pinning only counted for renewal time larger than $N^{2a}$. This model does not match the setting in \cite{CSZ17b} completely. Nevertheless, with some slight effort, we find that the method in \cite{CSZ17b} can also be applied to the model.

The details for (a) $d=1, a\in(0,\frac{1}{2})$, (b) $d=2, a\in[\frac{1}{2},1]$ and (c) $d=2, a\in(0,\frac{1}{2})$ are slightly different and hence we will separate the proofs for each case (The case $d=1,a=0$ has been sketched in Remark \ref{R1}). The first case is the most lengthy one, for which we give full details, where many computations can be reused in the other two cases.

\vspace{0.15cm}
\noindent\textbf{(a) $\bf d=1, a\in(0,\frac{1}{2})$.} As we mentioned above, our strategy is to approximate $\hat{Z}_{N,k}$ by $Z_{N,k}$. Recall that in this case, $\beta_N=\hat{\beta}\sqrt{\pi/2(1-2a)RN^a\log N}$.

Since $a\in(0,\frac{1}{2})$, we can find some $\epsilon>0$ such that $2a+\epsilon<1$. Then we write
\begin{equation}\label{indep}
Z_{N,k}=Z_{N,k}^{<}+Z_{N,k}^{\geq},
\end{equation}
where
\begin{equation*}
Z_{N,k}^{<}=\sum\limits_{\cA_{N,k}^{<}}\prod\limits_{j=1}^{k}p_{n_j-n_{j-1}}(x_j-x_{j-1})\xi_{n_j,x_j}\quad\text{and}\quad Z_{N,k}^{\geq}=\sum\limits_{\cA_{N,k}^{\geq}}\prod\limits_{j=1}^{k}p_{n_j-n_{j-1}}(x_j-x_{j-1})\xi_{n_j,x_j}
\end{equation*}
with
\begin{equation*}
\sum\limits_{\cA_{N,k}^{<}}:=\sum\limits_{1\leq n_1<\cdots<n_k\leq N, \atop\exists j, n_j-n_{j-1}< N^{2a+\epsilon}}\sum\limits_{(x_1,\cdots,x_k)\in B(RN^a)^{\bigotimes k}}\quad\text{and}\quad\sum\limits_{\cA_{N,k}^{\geq}}:=\sum\limits_{1\leq n_1<\cdots<n_k\leq N, \atop\forall j, n_j-n_{j-1}\geq N^{2a+\epsilon}}\sum\limits_{(x_1,\cdots,x_k)\in B(RN^a)^{\bigotimes k}}.
\end{equation*}
We claim that $Z_{N,k}^<$ can be made arbitrarily small by choosing $\epsilon$ small enough.

Since for any time sequence $(n_1,\cdots,n_k)$ in $\cA_{N,k}^{<}$ and $(m_1,\cdots,m_k)$ in $\cA_{N,k}^{\geq}$, there must be some $j$ such that $n_j-n_{j-1}<m_j-m_{j-1}$ and hence $\bbE[Z_{N,k}^{<}Z_{N,k}^{\geq}]=0$ by the independence of $\xi$. Therefore, to show that $Z_{N,k}^<$ is negligible, we only need to compare their second moments. We have that
\begin{equation*}
\var(Z_{N,k}^{<})=\sum\limits_{\cA_{N,k}^{<}}\prod\limits_{j=1}^{k}p_{n_j-n_{j-1}}^2(x_j-x_{j-1})\bbE\big[\xi_{n_j,x_j}^2\big],\quad
\var(Z_{N,k}^{\geq})=\sum\limits_{\cA_{N,k}^{\geq}}\prod\limits_{j=1}^{k}p_{n_j-n_{j-1}}^2(x_j-x_{j-1})\bbE\big[\xi_{n_j,x_j}^2\big].
\end{equation*}

For $\var(Z_{N,k}^{<})$, by recalling $I_N^{a,R}=\sum_{n=1}^N\bP(S_n=S_n',|S_n|\leq RN^a)$ from \eqref{i_n}, we have that
\begin{equation}\label{varZ<bound}
\begin{split}
&\var(Z_{N,k}^{<})\big/\bbE\left[\xi_{1,0}^2\right]^k\leq\sum\limits_{1\leq n_1-n_0,\cdots, n_k-n_{k-1}\leq N\atop\exists j, n_j-n_{j-1}< N^{2a+\epsilon}}\sum\limits_{(x_1,\cdots,x_k)\in B(RN^a)^{\bigotimes k}}\prod\limits_{j=1}^{k}p_{n_j-n_{j-1}}^2(x_j-x_{j-1})\\
\leq&\sum\limits_{j=1}^{k}\binom{k}{j}\Big(I_N^{a,2R}\Big)^{k-j}\bigg(\sum_{n=1}^{N^{2a+\epsilon}}\bP(S_n=S'_n,|S_n|\leq 2RN^a)\bigg)^{j}\leq\epsilon C_{K,a,R}\left(N^a\log N\right)^k,
\end{split}
\end{equation}
where in the last inequality, we follow the same lines in proving Lemma \ref{L:I_N} to get that for $d=1,a\in(0,\frac{1}{2})$, $\sum_{n=1}^{N^{2a+\epsilon}}\bP(S_n=S'_n,|S_n|\leq 2RN^a)\sim4\epsilon RN^a\log N$ as $N\to\infty$ and $C_{K,a,R}>0$ is a uniform constant for $k=1,\cdots,K$ and independent of $N$.

For $\var(Z_{N,k}^{\geq})$, we have that
\begin{equation}\label{varZ>bound}
\begin{split}
&\var(Z_{N,k}^{\geq})\big/\bbE\left[\xi_{1,0}^2\right]^k\geq\sum\limits_{N^{2a+\epsilon}\leq n_1-n_0,\cdots,n_k-n_{k-1}\leq N/k\atop(x_1-x_0,\cdots,x_k-x_{k-1})\in B(RN^a/k)^{\bigotimes k}}\prod\limits_{j=1}^{k}p_{n_j-n_{j-1}}^2(x_j-x_{j-1})\\
=&\bigg(\sum\limits_{n=N^{2a+\epsilon}}^{N/k}\bP(S_n=S'_n,|S_n|\leq RN^a/k)\bigg)^k\overset{N\to\infty}{\sim}\left(\frac{2(1-2a-\epsilon)R}{k\pi}N^a\log N\right)^k,
\end{split}
\end{equation}
where the last asymptotic behavior is uniform for $k=1,\cdots,K$ (check the proof of Lemma \ref{L:I_N}). Note that $\bbE[\xi_{1,0}^2]=1+o_N(1)$. Hence, the contribution from $Z_{N,k}^{<}$ can be made arbitrarily small by choosing small enough $\epsilon$, so we only need to deal with $Z_{N,k}^{\geq}$ in the following.

We now show that $\hat{Z}_{N,k}$ can be approximated by $Z_{N,k}^{\geq}$. Since $\forall j, |x_j-x_{j-1}|\leq 2RN^a$ and $n_j-n_{j-1}\geq N^{2a+\epsilon}$, by the local limit theorem, we have that $p_{n_j-n_{j-1}}(x_j-x_{j-1})\sim(2/\pi(n_j-n_{j-1}))^{1/2}$, uniformly for all $j$ with $x_j-x_{j-1}\leftrightarrow n_j-n_{j-1}$. Then we have
\begin{equation}\label{Z>=}
Z_{N,k}^{\geq}=\sum\limits_{1\leq n_1<\cdots<n_k\leq N\atop\forall j, n_j-n_{j-1}\geq N^{2a+\epsilon}}\sum\limits_{(x_1,\cdots,x_k)\in B(RN^a)^{\bigotimes k}\atop\forall j, x_j\leftrightarrow n_j}\prod\limits_{j=1}^{k}(1+o_N(1))\sqrt{\frac{2}{\pi(n_j-n_{j-1})}}\prod\limits_{j=1}^{k}\xi_{n_j,x_j},
\end{equation}
where we replace the restriction $x_j-x_{j-1}\leftrightarrow n_j-n_{j-1}$ by $x_j\leftrightarrow n_j$ by simple induction.
	
Note that by $n_j\geq n_j-n_{j-1}\geq N^{2a+\epsilon}\gg RN^a$, for any $x\in B(RN^a)$ with $x\leftrightarrow n_j$, $x$ is reachable by the simple random walk at time $n_j$. For the same reason, we have that $\{x\in B(RN^a): x\leftrightarrow n_j\}\cap\{x\in B(RN^a): x\leftrightarrow n_j+1\}=\emptyset$ and $\{x\in B(RN^a): x\leftrightarrow n_j\}\cup\{x\in B(RN^a): x\leftrightarrow n_j+1\}=B(RN^a)$. Hence, by writing $\Delta n_j=n_j-n_{j-1}$ and $\Delta x_j=x_j-x_{j-1}$ and assuming $\Delta n_j$ is even, we have the following combination 
\begin{equation*}
\begin{split}
&\sqrt{\frac{2}{\pi(\Delta n_j-1)}}\sum\limits_{|x_j|\leq RN^a\atop x_j\leftrightarrow n_j-1}\xi_{n_j-1,x_j}+\sqrt{\frac{2}{\pi\Delta n_j}}\sum\limits_{|x_j|\leq RN^a\atop x_j\leftrightarrow n_j}\xi_{n_j,x_j}\\
=&(1+o_N(1))p_{\Delta n_j}(0)\left(\sum\limits_{|x_j|\leq RN^a\atop x_j\leftrightarrow n_j-1}\xi_{n_j-1,x_j}+\sum\limits_{|x_j|\leq RN^a\atop x_j\leftrightarrow n_j}\xi_{n_j,x_j}\right).
\end{split}
\end{equation*}

For all even integers $n\geq2$, we introduce
\begin{equation}
\zeta_n:=\zeta_n^{(N)}=\frac{1}{\sqrt{2RN^a}}\left(\sum_{|x|\leq RN^a\atop x\leftrightarrow n-1}\xi_{n-1,x}+\sum_{|x|\leq RN^a\atop x\leftrightarrow n}\xi_{n,x}\right).
\end{equation}
Intuitively, we ``glue'' all reachable disorders for two consecutive times. Now we have
\begin{equation}\label{variantsummand}
\sum\limits_{n_j=n_{j-1}+N^{2a+\epsilon}}^N\sum\limits_{|x_j|\leq RN^a}p_{n_j-n_{j-1}}(x_j-x_{j-1})\xi_{n_j,x_j}=\sqrt{2RN^a}\sum\limits_{n_j=n_{j-1}+N^{2a+\epsilon}}^N(1+o_N(1))p_{n_j-n_{j-1}}(0)\zeta_{n_j}.
\end{equation}
Note that by the periodicity of the simple random walk, when writing \eqref{Z>=} by \eqref{variantsummand}, the summand is non-zero only if all $n_j$'s are even, and thus $\zeta_{n_j}$'s in \eqref{variantsummand} are well-defined.

Now we can write
\begin{equation}\label{approx}
\begin{split}
Z_{N,k}^{\geq}=(2RN^a)^{\frac{k}{2}}\sum\limits_{1\leq n_1<\cdots<n_k\leq N\atop\forall j, n_j-n_{j-1}\geq N^{2a+\epsilon}}\prod\limits_{j=1}^{k}(1+o_N(1))p_{n_j-n_{j-1}}(0)\zeta_{n_j}.
\end{split}
\end{equation}

Note that $\bbE[\zeta_n]=0$, $\var(\zeta_n)\overset{N\to\infty}{\longrightarrow}1$. By \cite[Section 5]{CSZ17b}, we can drop all $o_N(1)$'s in \eqref{approx} and replace $\xi_{n,0}$ by $\zeta_n$ in $\hat{Z}_{N,k}$ without changing the weak limit. Hence, in the following, we simply write
\begin{equation*}
Z_{N,k}^\geq=(2RN^a)^{\frac{k}{2}}\sum\limits_{1\leq n_1<\cdots<n_k\leq N\atop\forall j, n_j-n_{j-1}\geq N^{2a+\epsilon}}\prod\limits_{j=1}^{k}p_{n_j-n_{j-1}}(0)\zeta_{n_j}\quad\text{and}\quad \hat{Z}_{N,k}=\sum\limits_{1\leq n_1<\cdots<n_k\leq N\atop\forall j, n_j-n_{j-1}\geq N^{2a}}\prod\limits_{j=1}^{k}p_{n_j-n_{j-1}}(0)\zeta_{n_j}.
\end{equation*}

We show that $\|\beta_N^k Z_{N,k}^{\geq}-\tilde{\beta}_N^k\hat{Z}_{N,k}\|_2\overset{N\to0}{\longrightarrow}0$, where $\tilde{\beta}_N=\hat{\beta}\sqrt{\pi/(1-2a)\log N}$. We adapt the method in \cite[Lemma 6.1]{CSZ17b} here to write.
\begin{equation}\label{compare}
\bbE\big[(\beta_N^k Z_{N,k}^{\geq}-\tilde{\beta}_{N}^k\hat{Z}_{N,k})^2\big]=\bbE\big[(\tilde{\beta}_N^k\hat{Z}_{N,k})^2\big]-\bbE\big[(\beta_N^k Z_{N,k}^{\geq})^2\big]-2\bbE\big[(\tilde{\beta}_N^k\hat{Z}_{N,k}-\beta_N^k Z_{N,k}^{\geq})\beta_N^k Z_{N,k}^{\geq}\big].
\end{equation}
First note that the terms $(2RN^a)^{k/2}$ in both $\beta_N$ and $Z_{N,k}^\geq$ cancel out. Then, reasoning as \eqref{varZ<bound} and \eqref{varZ>bound}, by enlarging the domain of $(n_1,\cdots,n_k)$ from $1\leq n_1\leq\cdots\leq n_k\leq N$ to $1\leq n_1-n_0,\cdots,n_k-n_{k-1}\leq N$ and shrink it to $1\leq n_1-n_0,\cdots,n_k-n_{k-1}\leq N/k$, the first two second moments on the right-hand side of \eqref{compare} both converge to $\hat{\beta}^{2k}$ as $N\to\infty$ and $\epsilon\to0$. Finally, in the last term,
\begin{equation*}
\tilde{\beta}_N^k\hat{Z}_{N,k}-\beta_N^k Z_{N,k}^{\geq}=\tilde{\beta}\sqrt{\frac{\pi}{(1-2a)\log N}}\sum\limits_{1\leq n_{1}<\cdots<n_k\leq N\atop\forall j, N^{2a}\leq n_j-n_{j-1}<N^{2a+\epsilon}}\prod_{j=1}^k p_{n_j-n_{j-1}}(0)\zeta_{n_j}.
\end{equation*}
By the independence of $\zeta_n$, the last expectation is $0$. Hence, we get that $\beta_N^k Z_{N,k}^{\geq}$ can be made arbitrarily close to $\tilde{\beta}_N^k\hat{Z}_{N,k}$ in $L_2$.

It remains to show that $Z_N^{(K)}:=1+\sum_{k=1}^K\tilde{\beta}_N^k\hat{Z}_{N,k}$ converges in distribution to the desired limit. We mimic the procedure in \cite[Section 4]{CSZ17b} by only pointing out some key steps.

Let us denote $\hat{Z}_{N,k}^{(1)}:=\hat{Z}_{N,k}$. By the same argument for \eqref{compare}, we can enlarge the domain of $(n_1,\cdots,n_k)$ to approximate $\hat{Z}_{N,k}^{(1)}$ in $L_2$ by
\begin{equation*}
\hat{Z}_{N,k}^{(2)}:=\sum\limits_{N^{2a}\leq n_1-n_{0},\cdots,n_k-n_{k-1}\leq N}\prod\limits_{j=1}^{k}p_{n_j-n_{j-1}}(0)\zeta_{n_j}.
\end{equation*}

Then, for any positive integer $M$, we introduce
\begin{equation*}
\hat{Z}_{N,k}^{(3)}:=\sum\limits_{N^{\frac{\lfloor 2aM\rfloor}{M}}\leq n_1-n_{0},\cdots,n_k-n_{k-1}\leq N}\prod\limits_{j=1}^{k}p_{n_j-n_{j-1}}(0)\zeta_{n_j}.
\end{equation*}
It is not hard to show that $((1-2a-o_M(1)-o_N(1))/(1-2a))^k\leq\var(\tilde\beta_N^k\hat{Z}_{N,k}^{(3)})\leq 1$. Hence, we can approximate $\hat{Z}_{N,k}^{(2)}$ by $\hat{Z}_{N,k}^{(3)}$ in $L_2$ as $N\to\infty$ and $M\to\infty$.

Now we can split the interval $(N^{\lfloor 2aM\rfloor/M},N]$ according to $M$. Let $i_M^a=\lfloor 2aM\rfloor$, then
\begin{equation}\label{splitinterval}
\Big(N^{\frac{\lfloor 2aM\rfloor}{M}},N\Big]=\bigcup\limits_{i=i_M^a}^{M-1}I_i:=\bigcup\limits_{i=i_M^a}^{M-1}\Big(N^{\frac{i}{M}},N^{\frac{i+1}{M}}\Big],
\end{equation}
and we can write
\begin{equation*}
\hat{Z}_{N,k}^{(3)}=\left(\frac{(1-2a)\log N}{\pi M}\right)^{\frac{k}{2}}\sum\limits_{i_M^a\leq i_1,\cdots,i_k\leq M-1}Z_{i_1,\cdots,i_k}^{N,M},
\end{equation*}
where
\begin{equation*}
Z_{i_1,\cdots,i_k}^{N,M}=\left(\frac{\pi M}{(1-2a)\log N}\right)^{\frac{k}{2}}\sum\limits_{n_1-n_0\in I_{i_1},\cdots,n_k-n_{k-1}\in I_{i_k}}\prod\limits_{j=1}^{k}p_{n_j-n_{j-1}}(0)\zeta_{n_j}.
\end{equation*}

Next, we perform a \textit{dominated sequence decomposition} as (4.9) in \cite{CSZ17b}. We call $\boldsymbol{i}=\{i_1,\cdots,i_\ell\}$ a dominated sequence if $i_1>i_2,\cdots,i_\ell$. Then any sequence $\{i_1,\cdots,i_k\}$ can be decomposed by consecutive dominated subsequences $\boldsymbol{i}^{(1)}=\{i_1,\cdots,i_{\ell_2}-1\},\cdots,\boldsymbol{i}^{(m)}=\{i_{\ell_m},\cdots,i_k\}$, where $m:=m(\boldsymbol{i})$ and $i_{\ell_1}=i_1<\cdots<i_{\ell_m}$. Steps \textbf{(A2)}-\textbf{(A3)} in \cite{CSZ17b} together yield that $\hat{Z}_{N,k}^{(3)}$ can be approximated in $L_2$ by
\begin{equation*}
\hat{Z}_{N,k}^{(4)}:=\left(\frac{(1-2a)\log N}{\pi M}\right)^{\frac{k}{2}}\sum\limits_{\boldsymbol{i}\in\{i_M^a,\cdots,M\}_\#^k}Z_{\boldsymbol{i}^{(1)}}^{N,M}\cdots Z_{\boldsymbol{i}^{(m)}}^{N,M},
\end{equation*}
where $\{i_M^a,\cdots,M\}_\#^k:=\{(i_1,\cdots,i_k)\in\{i_M^a,\cdots,M\}^{\bigotimes k}: |i_j-i_{j'}|\geq 2, \forall j\neq j'\}$.

Finally, by steps \textbf{(K)}-\textbf{(A4)} in \cite{CSZ17b}, Let $N\to\infty$, and then $K\to\infty$, and lastly $M\to\infty$
\begin{equation*}
\begin{split}
&1+\sum\limits_{k=1}^K\tilde\beta_N^k\hat{Z}_{N,k}^{(4)}\overset{d}{\longrightarrow}1+\sum\limits_{k=1}^\infty\int\cdots\int_{2a<t_1<\cdots<t_k<1}\prod\limits_{j=1}^k\frac{\hat{\beta}}{\sqrt{1-\hat{\beta}^2 t_j}}\di W_{t_j}\\
=&:\exp\left(\int_{2a}^1\frac{\hat{\beta}}{\sqrt{1-\hat{\beta}^2 t}}\di W_t\right):=\exp\left(\int_{2a}^1\frac{\hat{\beta}}{\sqrt{1-\hat{\beta}^2 t}}\di W_t-\frac{1}{2}\int_{2a}^1\frac{\hat{\beta}^2}{1-\hat{\beta}^2 t}\di t\right),
\end{split}
\end{equation*}
which has the same distribution as $\bZ_{\hat{\beta}}$ in \eqref{conv2}.

\vspace{0.15cm}
\noindent\textbf{(b) $\bf d=2, a\in[\frac{1}{2},1]$.} Recall that $\beta_N=\hat{\beta}\sqrt{\frac{\pi}{\log N}}$. 
Similar to \eqref{compare}, we have
\begin{equation*}
\bbE\big[\beta_N^{2k}( Z_{N,k}-\hat{Z}_{N,k})^2\big]=\bbE\big[(\beta_N^k\hat{Z}_{N,k})^2\big]-\bbE\big[(\beta_N^k Z_{N,k})^2\big]-2\beta_N^{2k}\bbE\big[(\hat{Z}_{N,k}-Z_{N,k})Z_{N,k}\big].
\end{equation*}
Reasoning as \eqref{compare}, the first two second moments on the right-hand side both converge to $\hat{\beta}^{2k}$ as $N\to\infty$. For the last term, note that 
\begin{equation*}
\hat{Z}_{N,k}-Z_{N,k}=\sum\limits_{1\leq n_1<\cdots<n_k\leq N}\sum\limits_{(x_1,\cdots,x_k)\notin B(RN^a)^{\bigotimes k}}\prod_{j=1}^k p_{n_j-n_{j-1}}(x_j-x_{j-1})\xi_{n_j,x_j}.
\end{equation*}
By the independence of $\xi_{n,x}$, this expectation is $0$. Recall \eqref{hatZ} and note that $\hat{Z}_{N,k}$ is the $k$-th term in the polynomial chaos expansion for the partition function of the classic $(1+2)$-dimensional directed polymer model. Then the result follows by \cite{CSZ17b}.

\vspace{0.15cm}
\noindent\textbf{(c) $\bf d=2, a\in(0,\frac{1}{2})$}. 
In this case, we show that $\|\beta_N^k(Z_{N,k}-\hat{Z}_{N^{2a},k})\|_2\overset{N\to\infty}{\longrightarrow}0$, where we now have $\beta_N=\hat{\beta}\sqrt{\frac{\pi}{2a\log N}}$. We write
\begin{equation*}
\bbE\big[\beta_N^{2k}( Z_{N,k}-\hat{Z}_{N^{2a},k})^2\big]=\bbE\big[(\beta_N^k\hat{Z}_{N^{2a},k})^2\big]-\bbE\big[(\beta_N^k Z_{N,k})^2\big]-2\beta_N^{2k}\bbE\big[(\hat{Z}_{N^{2a},k}-Z_{N,k})Z_{N,k}\big].
\end{equation*}
Reasoning as the case (b) above, the first two terms both go to $\hat{\beta}^{2k}$ as $N\to\infty$. However, the last term is not $0$ since $2a<1$ and thus the summands in $Z_{N,k}$ are no longer a subset of those in $\hat{Z}_{N^{2a},k}$. By independence of $\xi_{n,x}$, we have that
\begin{equation*}
\bbE\big[(\hat{Z}_{N^{2a},k}-Z_{N,k})Z_{N,k}\big]=\bbE\big[(\tilde{Z}_{N
^{2a},k})^2\big]-\bbE\big[(Z_{N,k})^2\big],
\end{equation*}
where
\begin{equation*}
\bbE\big[(\tilde{Z}_{N^{2a},k})^2\big]\sim\sum\limits_{1\leq n_1<\cdots<n_k\leq N^{2a}}\sum\limits_{(x_1,\cdots,x_k)\in B(RN^a)^{\bigotimes k}}\prod\limits_{j=1}^{k}p_{n_j-n_{j-1}}^2(x_j-x_{j-1}).
\end{equation*}
Again, by the argument in the proof of Lemma \ref{L:I_N}, and the techniques that enlarging $1\leq n_1<\cdots<n_k\leq N^{2a}$ to $1\leq n_1-n_0,\cdots,n_k-n_{k-1}\leq N^{2a}$ and shrinking it to $1\leq n_1-n_0,\cdots,n_k-n_{k-1}\leq N^{2a}/k$, we have that $\bbE[(\tilde{Z}_{N^{2a},k})^2]\sim(2a\log N/\pi)^k$ and thus the last term converges to $0$. Recall the definition of $\hat{Z}_{N^{2a},k}$ from \eqref{hatZ} and we conclude the last the case.

\vspace{0.15cm}
\noindent{\bf (ii) Case $\bf\hat{\beta}\geq1$.} The proof is identical to that in \cite{CSZ17b}. We sketch it for completeness.

It is enough to show that for any $\theta\in(0,1)$, $\bbE[(Z_N)^\theta]$ converges to $0$. We have that
\begin{equation*}
\frac{\dd}{\dd\beta}\bbE\left[(Z_{N,\beta}^{\omega})^\theta\right]=\theta\sum\limits_{n=1}^{N}\bE\bigg[\mathbbm{1}_{\{|S_n|\leq RN^a\}}\bbE_S\big[(\omega_{n,S_n}-\lambda'(\beta))(Z_{N,\beta}^{\omega})^{\theta-1}\big]\bigg],
\end{equation*}
where $\bbP_S$ is a probability measure with $\bbE_S[X]=\bbE[\exp(\sum_{n=1}^N(\beta\omega_{n,S_n}-\lambda(\beta)))X]$. Then by the FKG inequality and using $\bbE_S[\omega_{n,S_n}-\lambda'(\beta)]=0$, we get that
\begin{equation*}
\frac{\dd}{\dd\beta}\bbE\left[(Z_{N,\beta}^{\omega})^\theta\right]\leq\theta\sum\limits_{n=1}^N\bE\bigg[\mathbbm{1}_{\{|S_n|\leq RN^a\}}\bbE_S\big[(\omega_{n,S_n}-\lambda'(\beta))\big]\bbE_S\left[(Z_{N,\beta}^{\omega})^{\theta-1}\right]\bigg]=0.
\end{equation*}
Thus, for any $\theta\in(0,1)$, $\bbE[(Z_{N,\beta}^{\omega})^\theta]$ is non-increasing in $\beta$. Then for any $\hat{\beta}'\geq1>\hat{\beta}$,
\begin{equation*}
\limsup\limits_{N\to\infty}\bbE\left[\big(Z_{N,\beta'_N}^{\omega}\big)^\theta\right]\leq\limsup\limits_{N\to\infty}\bbE\left[\big(Z_{N,\beta_N}^{\omega}\big)^\theta\right]=\bbE\left[(\bZ_{\hat{\beta}})^\theta\right]=(1-\hat{\beta}^2)^{\frac{\theta(1-\theta)}{2}},
\end{equation*}
where the equality is due to Theorem \ref{T1} for $\hat{\beta}<1$. Let $\hat{\beta}\uparrow1$ and the proof is completed.

\subsection{Remarks for the result}
Finally, we make the following remarks for our results.
\begin{remark}\label{r:region}
We could add some slowly varying function in the growth rate of the width of the environment region. However, the results do not change when the parameters of the system are not critical, namely, the system is still disorder relevant if $d=1, a\in(\frac{1}{2},1]$, marginally relevant if $d=1,a\in[0,\frac{1}{2})$ and $d=2, a\in(0,1]$ and disorder irrelevant if $d\geq3, a\in[0,1]$. The situations at the critical points are more subtle. Generally speaking, if the slowly function $L(\cdot)$ oscillates between $0$ and $+\infty$, then the weak coupling limit is hard to expect.

To move a bit further, if $d=1, a=\frac{1}{2}$, and the width of the environment region is of order $\sqrt{N}/L(N)$, then the model should be disorder relevant if $L(\cdot)$ is bounded above while marginally relevant if $L(N)$ is increasing to $+\infty$, since it can be shown that the intersection time for the latter case is $I_N\sim4R\log L(N)\sqrt{N}/L(N)$. If $d=2, a=0$, and the width of the environment region is of order $L(N)$ with $L(N)$ increasing to $+\infty$, then the model should also be marginally relevant, since $I_N\sim2\log L(n)/\pi$. We choose to treat pure power law in this paper to avoid lengthy notations and tedious computations.
\end{remark}

\begin{remark}\label{r:cone}
As we mentioned in Subsection \ref{S14}, one may consider the model whose disorders are placed in the cone $\Omega_N:=\{(n,x):~1\leq n\leq N, |x|\leq Rn^a\}$. In this case, one can show that, as Lemma \ref{L:I_N},
\begin{equation*}
I_{N}\sim\begin{cases}
C_{1,a,R}\sqrt{N},\quad&\mbox{for}~d=1,a\in[\frac{1}{2},1], R>0,\\[5pt]
\frac{2R}{a\pi}N^{a},\quad&\mbox{for}~d=1,a\in(0,\frac{1}{2}), R>0,\\[5pt]
\frac{2\lfloor R\rfloor+1}{\pi}\log N,\quad&\mbox{for}~d=1, a=0, R\geq0,\\[5pt]
C_{2,a,R}\log N,\quad&\mbox{for}~d=2, a\in[\frac{1}{2},1], R>0,\\[5pt]
\sum\limits_{n=1}^{\infty}\bP^{\bigotimes2}(S_n=S'_n,|S_n|\leq Rn^a),\quad&\mbox{for}~d=2, a\in[0,\frac12),R\geq0~\mbox{or}~d\geq3, a\in[0,1], R\geq0,\\[5pt]
\end{cases}
\end{equation*}
where
\begin{equation*}
C_{1,a,R}=\begin{cases}
\frac{2}{\sqrt{\pi}},~&\text{for}~a\in(\frac{1}{2},1],\\[5pt]
\frac{2}{\pi}\int_{|x|\leq R}e^{-x^2}\di x,~&\text{for}~a=\frac{1}{2},
\end{cases}
~~\text{and}~~ C_{2,a,R}=\begin{cases}
\frac{1}{\pi},~&\text{for}~a\in(\frac{1}{2},1],\\[5pt]
\frac{1}{\pi}(1-e^{-2R^2}),~&\text{for}~a=\frac{1}{2}.
\end{cases}
\end{equation*}

It is not hard to deduce that the system is disorder relevant for $d=1, a\in[\frac{1}{2},1]$, marginally relevant for $d=1, a=0$ and $d=2, a\in[\frac{1}{2},1]$, and disorder irrelevant for $d=2, a\in[0,\frac{1}{2})$ and $d\geq3, a\in[0,1]$. However, the case $d=1, a\in(0,\frac{1}{2})$ is much more sophisticated.

The reason is that in this case, the second moment calculation is too rough to determine the phase of the system. To be specific, when $d=1, a\in(0,\frac{1}{2})$, the main contribution to the $k$-th order term of the partition function comes from
\begin{equation*}
\sum\limits_{1\leq n_1<\cdots<n_k\leq N\atop\forall j, n_j-n_{j-1}\geq\epsilon N}\sum\limits_{\forall j, |x_j|\leq Rn_j^a}\prod\limits_{j=1}^k p_{n_j-n_{j-1}}(x_j-x_{j-1})\xi_{n_j,x_j},
\end{equation*}
which, when multiplied by $(a\pi/2R N^a)^k$, converges to
\begin{equation*}
\int_{0<t_1<\cdots<t_k<1\atop\forall j, t_j-t_{j-1}>\epsilon}\prod\limits_{j=1}^{k}\frac{t_j^a}{t_j-t_{j-1}}\di t_j,
\end{equation*}
as $N\to\infty$. Note that for this integral, it is only valid to send $\epsilon\to0$ for $k=1$. Thus, the coupling $\beta_N:=\hat{\beta}a\pi/2R N^a$ is too strong.

To explain this phenomenon, considering the first return to the environment region, if it happens after time $O(N)$, then for the rest of time, the width of the environment region is of order $N^a$, which implies that the model in cone environments is comparable to the model in tube environments, and thus there should be some logarithmic correction for the second moment. However, before the first return to the cone, the width of the cone is too thin (sub-diffusive), and therefore the logarithmic term disappears in the first term of the chaos expansion.

One reasonable conjecture for $d=1, a\in(0,\frac{1}{2})$ should be the following: let $\beta_N:=\hat{\beta}a\pi/2R N^a\log N$, then $\log N(Z_N-1)$ converges in distribution to $\bZ_{\hat{\beta}}-1$ with $\bZ_{\hat{\beta}}$ in \eqref{conv2} for $\hat{\beta}\in(0,1)$, while $0$ for $\hat{\beta}\geq1$. This conjecture may be proved by an adaption of the method in \cite{CSZ17b}. To keep the paper in a reasonable length, we treat the model in tube environments, which is technically simpler to illustrate the phase transition from the pinning model to the classic directed polymer model more transparently.
\end{remark}

\subsection*{Acknowledgements} Ran Wei is supported by a public grant overseen by the French National Research Agency, ANR SWiWS (ANR-17-CE40-0032-02). 
Jinjiong Yu is supported by NSFC 12101238.
We would like to thank Rongfeng Sun and Quentin Berger for helpful suggestions and discussions. We also thank Francesco Caravenna and Nikos Zygouras for telling us some references. Finally, we thank the anonymous referees, who help us greatly improve the quality of the paper and correct some mistakes.

\bibliographystyle{plain}
\bibliography{references}
\end{document}